\documentclass[reqno,12pt]{amsart}
\usepackage{latexsym,amsmath,amssymb,amsbsy,amscd,bm,amsfonts}
\usepackage{hyperref}
\hypersetup{
  colorlinks   = true, 
  urlcolor     = blue, 
  linkcolor    = blue, 
  citecolor    = red   
}
\theoremstyle{plain}
\newtheorem{theorem}{Theorem}[section]

\newtheorem{cor}{Corollary}[section]

\theoremstyle{definition}
\newtheorem{definition}{Definition}[section]

\theoremstyle{remark}

\numberwithin{equation}{section}

\newcommand{\real}{{\mathbb R}}

\newcommand{\comp}{{\mathbb C}}

\newcommand{\G}{{\mathbb G}}
\newcommand{\Lie}{{\mathcal G}}
\newcommand{\T}{{\mathcal T}}
\newcommand{\To}{{\mathbb T}}
\newcommand{\Ho}{{\mathcal H}}

\newcommand{\rexp}{\overrightarrow{\exp}}
\newcommand{\lexp}{\overleftarrow{\exp}}

\DeclareMathOperator{\hexp}{exp}

 \DeclareMathOperator{\Ad}{Ad}
 \DeclareMathOperator{\ad}{ad}

 \DeclareMathOperator{\Aut}{Aut}
 \DeclareMathOperator{\End}{End} 
 
 \DeclareMathOperator{\tr}{trace}

  \DeclareMathOperator{\SU}{SU}

  \DeclareMathOperator{\GL}{GL}

\begin{document}
\title[Cartan exponential]{The Cartan exponential map in semi-simple, compact Lie groups}
\author{Andr\'{a}s Domokos}
\address{Department of Mathematics and Statistics,
California State University Sacramento, 6000 J Street, Sacramento, CA, 95819, USA}
\email{domokos@csus.edu}

\date{\today}

\keywords{semi-simple, compact Lie groups, exponential map, sub-Riemannian geometry}

\subjclass[2010]{53C22, 22C05}

\begin{abstract}
In this paper we introduce a new type of exponential map in semi-simple compact Lie groups, which is related to the sub-Riemannian geometry generated by the orthogonal complement of a Cartan subalgebra in a similar way to how the group exponential map is related to the Riemannian geometry. \end{abstract}

\maketitle
\section{Introduction}
This paper is the result of our efforts to understand the role of the chronological exponential \cite{agra2} in the geometry of Lie groups, in general, and of semi-simple, compact Lie groups, in particular. We will try to find connections similar to those between the group exponential map and the Riemannian geometry of the Lie group.\\

For the beginning, let  $\G$ be a matrix Lie group, which is a closed subgroup of $\GL (n , \comp )$ and $\Lie$ be its Lie algebra equipped with an inner product which will be specified later.

The group exponential map will be denoted  by $\exp : \Lie  \to \G$, where $\exp(X) = e^X$ is the usual matrix exponential.

We identify $\Lie$ with the Lie algebra of left-invariant vector fields and consider the tangent space at $g \in \G$ to be $T_g \G = g \cdot \Lie$. 

 For a function $X : \real \to \Lie$, consider the non-autonomous dynamical system
 \begin{equation}\label{dyn}
 \left\{ \begin{array}{l} \gamma' (t) = \gamma (t) \cdot X(t)\\ \gamma (0) = I \, . \end{array} \right.
 \end{equation}

Regarding the existence and uniqueness of solutions to \eqref{dyn}, which is a system of linear differential equations, we have the following result.

\begin{theorem}\label{thmexistence}
If $X : \real \to \Lie$ is continuous, then \eqref{dyn} has a unique solution $\gamma : \real \to \G$.
Moreover, if $X$ is smooth, then $\gamma$ is smooth.
\end{theorem}

The solution $\gamma$, which is called the flow of $X$, is a special case of the right chronological exponential flow \cite{agra2}. Therefore, we will adopt the notation from \cite {agra2}:
$$\gamma (t)  = \rexp \int_0^t X(s) ds \, .$$

It is worth mentioning that $\rexp \int_0^t X(s) ds = e^{\int_0^t X(s) ds}$ if and only if $[X(s) , X(t) ] =0$, for all $s,t \in \real$. In particular, if $X(t) \equiv X \in \Lie$, then 
$\rexp \int_0^t X(s) ds = e^{t X}$.

If we rewrite \eqref{dyn} as
$$\gamma (t) = I + \int_0^t \gamma (s) \cdot X(s) ds \, ,$$
and iterate it, we get the formal series expansion of the right chronological exponential flow
\begin{equation}\label{series}
\rexp \int_0^t X(s) ds = I + \sum_{k=1}^{\infty} \int_{\Delta_k (t)} X(s_k) \cdots   X(s_1 ) ds_k \dots ds_1 \, ,
\end{equation}
where
$$\Delta_k (t) = \left\{ (s_1, \dots s_k ) \in \real^k \, : \, 0 \leq s_k \leq \dots \leq s_1 \leq t \right\} \, .$$
By the fact that
$$ \left\| \int_{\Delta_k (t)} X(s_k) \cdots   X(s_1 ) ds_k \dots ds_1 \right\| \leq \frac{1}{k!} \left( \int_0^t ||X(s) || ds \right)^k ,$$ 
it follows that the series \eqref{series} converges pointwise under relatively mild conditions, like $||X(t)|| \leq M t^p$.\\

The inverse flow $\gamma^{-1} (t)$ is defined by the equation $\gamma^{-1} (t) \cdot \gamma (t) = I$. By differentiating this equation, we get that the inverse flow is a solution of the dynamical system
 \begin{equation}\label{dyninv}
 \left\{ \begin{array}{l} \left( \gamma^{-1} \right)' (t) = -X(t) \cdot \gamma^{-1} (t)\\ \gamma^{-1} (0) = I \, . \end{array} \right.
 \end{equation}
Hence, for the notation of the inverse flow we use the left chronological exponential notation
$$\gamma^{-1} (t) = \lexp \int_0^t \hspace{-0.2cm} -X(s) ds \, .$$ 
The flow starting at time $t_0$ can be defined as
$$\rexp \int_{t_0}^t X(s) ds = \lexp \int_0^{t_0}\hspace{-0.2cm} -X(s) ds \cdot \rexp \int_0^t X(s) ds \, ,$$ 
and hence we have the natural equation regarding the composition of flows
$$\rexp \int_{t_0}^{t_1} X(s) ds \cdot \rexp \int_{t_1}^{t_2} X(s) ds =  
\rexp \int_{t_0}^{t_2} X(s) ds \, .$$

For the following part of this paper we need the adjoint representation of $\G$,
$$\Ad \colon \G \to \Aut (\Lie ) \, , \; \; \Ad (g) (X) = g X g^{-1} \, ,$$
and its differential at the identity, which is the adjoint representation of $\Lie$,
$$\ad \colon \Lie \to \End (\Lie) \, , \; \; \ad X (Y) = [X,Y] \, .$$ 

Regarding the flow of $X(t)+Y(t)$ we have the following result, which is called the Variations Formula in \cite{agra2}. For the clarity of the presentation of our ideas, we provide a proof.

\begin{theorem}\label{thmsumflow}
Let $X, Y : \real \to \Lie$ be continuous. Then
\begin{multline}\label{eqsumflow}
\rexp \int_0^t X(s) +Y(s)ds \\ = \rexp \int_0^t \Ad \left( \rexp \int_0^s X(s_1) ds_1 \right) (Y(s)) ds \cdot \rexp \int_0^t X(s) ds \, .
\end{multline}
\end{theorem} 

\begin{proof}
Let us use the notations
$$F(t) = \rexp \int_0^t X(s)+Y(s) ds \, , \; \Phi (t) =  \rexp \int_0^t X(s) ds \, .$$
We want to find a flow $\Psi (t)$ such that
$$F(t) = \Psi (t) \cdot \Phi (t) \, .$$
By differentiation, we get that
$$F'(t) = \Psi'(t) \cdot \Phi (t) + \Psi (t) \cdot \Phi'(t) = F(t) \cdot (X(t) +Y(t)) \, .$$
Hence,
$$\Psi'(t) \cdot \Phi(t) + \Psi(t) \cdot \Phi(t) \cdot X(t) = \Psi(t) \cdot \Phi(t) \cdot X(t)+ \Psi (t) \cdot \Phi (t) \cdot Y(t) \, ,$$ 
which leads to
$$\Psi'(t) = \Psi(t) \cdot \left( \Phi(t) \cdot Y(t) \cdot \Phi^{-1} (t) \right) \, .$$
Therefore,
$$\Psi(t) = \rexp \int_0^t \Ad (\Phi(s))(Y(s) ds \, ,$$
from which \eqref{eqsumflow} follows.
\end{proof}

By considering two constant vector fields $X$ and $Y$, formula \eqref{eqsumflow} immediately implies the following corollary.
 
\begin{cor}\label{corsumflow}
If $X, Y \in \Lie$, then
\begin{equation}\label{eqsumflow1}
e^{t(X+Y)} \cdot e^{-tX}= \rexp \int_0^t e^{sX} \cdot Y \cdot e^{-sX} ds  \, .
\end{equation}
\end{cor}

As an interesting fact, by considering $X+Y =Z$ and $W =-X$ in \eqref{eqsumflow1}, we have the following form of the Baker-Campbell-Hausdorff formula:

\begin{cor}
If $Z, W \in \Lie$, then
\begin{equation}\label{BCH}
e^{Z} \cdot e^{W}= \rexp \int_0^1  e^{-sW} \cdot (Z +W) \cdot e^{sW}  ds  \, .
\end{equation}
\end{cor}

Moreover, we can use \eqref{eqsumflow1} and \eqref{series} to find the differential of $\exp$ at $X$ in the direction of $Y$. For a different proof, see \cite[Theorem 1.5.3]{dui}.

\begin{cor}\label{cordifferential}
\begin{equation}\label{eqdifferential}
D(\exp)_X (Y) = \int_0^1 \Ad \left( e^{sX} \right) (Y) \, ds \cdot e^X \, .
\end{equation}
\end{cor}

\begin{proof}
\begin{equation*}
\begin{split}
D(\exp)_X (Y) & = \lim_{\varepsilon \to 0} \, \frac{e^{X+\varepsilon Y} - e^X}{\varepsilon}\\
&= \lim_{\varepsilon \to 0} \, \frac{\left( e^{X+\varepsilon Y} \cdot e^{-X} - I \right) \cdot e^X}{\varepsilon}\\
& = \lim_{\varepsilon \to 0} \, \frac{1}{\varepsilon} \left( \rexp \int_0^1 e^{sX} \cdot (\varepsilon Y) \cdot e^{-sX} \, ds  - I \right) \cdot e^X \\
& = \lim_{\varepsilon \to 0} \, \frac{1}{\varepsilon} \left( \varepsilon \int_0^1 e^{sX} \cdot Y \cdot e^{-sX} \, ds  + O(\varepsilon^2 ) \right) \cdot e^X \\
& = \int_0^1 \Ad \left( e^{sX} \right) (Y) \, ds \cdot e^X \, .
\end{split}
\end{equation*}
\end{proof}

\section{The Cartan exponential on semi-simple, compact Lie groups}
In this section $\G$ is a semi-simple, compact, connected Lie group.
The Killing form  
$$K(X,Y) = \tr (\ad X \cdot \ad Y ) \, ,$$
is negative definite and non-degenerate on the Lie algebra of a semi-simple, compact Lie group, and hence we can define an inner product on $\Lie$ as
\begin{equation}\label{eq:InnerProduct}
\langle X , Y \rangle = - \rho \, K(X,Y) \, ,
\end{equation}
where $\rho> 0$ is a constant, which can be adjusted according to our normalization preferences. 

For example, in $\SU (3)$ we have  $K(X,Y) = 6 \tr (X  \cdot Y)$, and the Gell-Mann matrices form an orthonormal basis of $\Lie$ with the choice of $\rho = \frac{1}{12}$, which gives $\langle X, Y \rangle = - \frac{1}{2} \tr (X \cdot Y)$.

The Killing form is $\Ad$-invariant, so $\Ad (g)$ is a unitary linear transformation of $\Lie$ for all $g \in \mathbb G$ and $\ad X$ is skew-symmetric for all $X \in \Lie$.

The inner product \eqref{eq:InnerProduct} generates a bi-invariant metric on $\G$, which implies the following properties of the Riemannian geometry on $\G$.

\begin{theorem} \cite[Section3.3]{arv}~\\
(a) All Riemannian geodesics through the identity of $\G$ are given by 
$$ \gamma_X (t) = e^{tX} , \; \; \text{where} \; \;  X \in \Lie .$$
Moreover, $\gamma_X$ has constant speed equal to $||X||$ and constant curvature of $\frac{1}{||X||^2} ||X^2||$.\\ 
(b) The differential of $\exp : \Lie \to \G$ at $0$ is the identity map of $\Lie$, which implies that $\exp$ is a local diffeomorphism at $0$. 
\end{theorem}

Let $\To$ be a maximal torus in $\G$ and $\T$ be its Lie algebra. In this case, $\T$ is a maximal commutative subalgebra of $\Lie$, called the Cartan subalgebra. Its dimension is called the rank of $\Lie$, and also the rank of $\G$. The orthogonal complement of $\T$ with respect to the inner product \eqref{eq:InnerProduct} is denoted by $\Ho$ and we call it the horizontal subspace of $\Lie$. The left translates $g \cdot \Ho$ of $\Ho$ form the horizontal distribution of a sub-Riemannian geometry on $\G$, which naturally combines the algebraical and topological properties of $\G$ \cite{mont}.\\


For the remainder of the paper, let us fix a maximal torus $\To$ and its Cartan subalgebra $\T$. \\

\begin{definition}
(a)  For a given $X = H+T \in \Ho \oplus \T$, the Cartan exponential flow $\hexp_{\T}(tX)$ is defined by
\begin{equation}
\hexp_{\T} (tX)= \rexp \int_0^t e^{sT} \cdot H \cdot e^{-sT} ds \, ,
\end{equation} 
or, equivalently by \eqref{BCH},
\begin{equation}
\hexp_{\T} (tX) = e^{tX} \cdot e^{-tT}  \, .
\end{equation} 
(b) The Cartan exponential map $\hexp_{\T} : \Lie \to \G$  defined by
\begin{equation}
\hexp_{\T} (X) = \rexp \int_0^1 e^{sT} \cdot H \cdot e^{-sT} ds \, , \; \text{if} \; X = H+T \in \Ho \oplus \T \, .
\end{equation} 
or, equivalently,
\begin{equation}
\hexp_{\T} (X) = e^{X} \cdot e^{-T} \, , \; \text{if} \; X = H+T \in \Ho \oplus \T \, .
\end{equation} 

\end{definition}

Regarding the properties of the Cartan exponential, we have the following theorem.\\

\begin{theorem}~\\
(a) All sub-Riemannian geodesics through the identity of $\G$ are given by
$\sigma_X (t) = \hexp_{\T} (tX)$, where $X \in \Lie.$\\
Moreover, $\sigma_X$ has a constant speed $||H||$ and a constant curvature
$$k = \frac{\left\Vert H^2 - [H,T] \right\Vert}{\| H \|^2} .$$
(b) The differential of the Cartan exponential map $\hexp_{\T} : \Lie \to \G$ at $0$ is
the orthogonal projection onto $\Ho$:
$$\left( D \hexp_{\T} \right)_0 (X) = H , \; \; \text{where} \; \; X = H+T \in \Ho \oplus \T .$$
\end{theorem} 

\begin{proof} For the fact that all sub-Riemannian geodesics through the identity of $\G$ are given by $\sigma_X (t) = \hexp_{\T} (tX)$, we quote \cite[Theorem 2.1]{dom17}.

Moreover, $\sigma_X' (t) = e^{tX} \cdot H \cdot e^{-tT}$ and the $\Ad$-invariance of the inner product of $\Lie$ implies that
\begin{equation*}
\begin{split}
\left\Vert \sigma_X' (t) \right\Vert_{\sigma_X (t)} &= \left\Vert \sigma_X (t)^{-1} \cdot \sigma_X' (t) \right\Vert\\
& = \left\Vert e^{tT} \cdot H \cdot e^{-tT} \right\Vert = ||H|| \, .
\end{split}
\end{equation*} 
The derivative of the unit tangent vector
$$V(t) = \frac{1}{||H||} e^{tX} \cdot H \cdot e^{-tT}$$
is
\begin{equation*}
\begin{split}
V'(t) &= \frac{1}{||H||} \left( e^{tX} \cdot (XH - HT) \cdot e^{-tT} \right) \\
& = \frac{1}{||H||} \left( e^{tX} \cdot (H^2 - [H,T]) \cdot e^{-tT} \right) \, ,
\end{split}
\end{equation*}
and hence, the curvature of $\sigma_X$ is
$$k(t) = \left. \frac{1}{||\sigma_X ' (t) ||} \left\Vert V' (t) \right\Vert \right|_{\sigma_X (t)} = \frac{\left\Vert H^2 - [H,T] \right\Vert}{\left\Vert H \right\Vert^2} \, .$$  

(b) For the differential of $\sigma_X (t) = \hexp_{\T} (tX)$ we have the formula: 
\begin{equation*}
\left( D \hexp_{\T} \right)_0 (X)  = \lim_{t \to 0} \sigma_X ' (t) 
 = \lim_{t \to 0} e^{tX} \, H \, e^{-tT} = H.
\end{equation*}
\end{proof}

\end{document}